\documentclass{article}
\usepackage{graphicx}
\usepackage{amsmath, amsthm,todonotes,verbatim}
\begin{document}
\setlength{\textwidth}{5.75in}

\makeatletter
\def\@citex[#1]#2{\if@filesw\immediate\write\@auxout{\string\citation{#2}}\fi
  \def\@citea{}\@cite{\@for\@citeb:=#2\do
    {\@citea\def\@citea{, }\@ifundefined
       {b@\@citeb}{{\bf ?}\@warning
       {Citation `\@citeb' on page \thepage \space undefined}}%
\hbox{\csname b@\@citeb\endcsname}}}{#1}}
\makeatother


\makeatletter
\def\doublespaced{\baselineskip=\normalbaselineskip	
    \multiply\baselineskip by 150			
    \divide\baselineskip by 100}			
\def\doublespace{\doublespaced}				
\makeatother

\makeatletter
\def\mitspaced{\baselineskip=\normalbaselineskip	
    \multiply\baselineskip by 115			
    \divide\baselineskip by 100}			
\def\mitspace{\mitspaced}
\makeatother

\makeatletter
\def\capamispaced{\baselineskip=\normalbaselineskip	
    \multiply\baselineskip by 105			
    \divide\baselineskip by 100}			
\def\capamispace{\capamispaced}
\makeatother

\makeatletter
\def\singlespaced{\baselineskip=\normalbaselineskip}	
\def\singlespace{\singlespaced}				
\makeatother

\makeatletter
\def\triplespaced{\baselineskip=\normalbaselineskip	
    \multiply\baselineskip by 3}			
\makeatother

\makeatletter
\def\widenspacing{\multiply\baselineskip by 125		
    \divide\baselineskip by 100}			
\def\whitespace{\widenspacing}				
\makeatother

\font\eightmsam=msam8
\font\ninemsam=msam9
\font\tenmsam=msam10
\newtheorem{Theorem}{Theorem}[section]
\newtheorem{Corollary}[Theorem]{Corollary}
\newtheorem{Proposition}[Theorem]{Proposition}
\newtheorem{Lemma}[Theorem]{Lemma}
\newtheorem{Claim}[Theorem]{Claim}
\newtheorem{Definition}[Theorem]{Definition}
\newtheorem{Example}[Theorem]{Example}
\newtheorem{Conjecture}[Theorem]{Conjecture}
\newtheorem{Question}[Theorem]{Question}
\newcommand{\BB}{{\ninemsam\char'04}}
\renewcommand{\Box}{\mbox{\BB}\medskip}
\newcommand{\optimal}{$\Theta(n^{1/2})$}
\newcommand{\joptimal}{$\Theta(n^{1/j})$}
\newcommand{\nfourth}{$\Theta(n^{1/4})$}
\newcommand{\nsixth}{$\Theta(n^{1/6})$}
\newcommand{\nfourthlog}{$\Theta(n^{1/4}\log(n))$}
\newcommand{\linear}{$\Theta(n)$}
\newcommand{\logn}{$\Theta(\log(n))$}

\newcommand{\N}{\mathbf N}
\newcommand{\Z}{\mathbf Z}
\newcommand{\R}{\mathbf R}
\setlength{\baselineskip}{0.3in}
\hoffset -0.4in
\doublespace
\title{Homotopy Relations for Digital Images
}

\author{Laurence Boxer
         \thanks{
    Department of Computer and Information Sciences,
    Niagara University,
    Niagara University, NY 14109, USA;
    and Department of Computer Science and Engineering,
    State University of New York at Buffalo.
    E-mail: boxer@niagara.edu
    }
\and{P. Christopher Staecker
     \thanks{
     Department of Mathematics, Fairfield 
     University, Fairfield, CT 06823-5195, USA.
     E-mail: cstaecker@fairfield.edu
     }
}
}

\date{ }
\maketitle

\begin{abstract}
We introduce three generalizations of homotopy equivalence 
in digital images, to allow us
to express whether a bounded and an unbounded digital image
with standard adjacencies
are similar with respect to homotopy.

We show that these three generalizations are not equivalent to ordinary homotopy equivalence, 
and give several examples. We show that, like homotopy equivalence, our three generalizations imply isomorphism of fundamental groups, and are preserved under wedges and Cartesian products.


\end{abstract}

\section{Introduction}
In {\em digital topology}, we study geometric and 
topological properties of digital images via tools 
adapted from geometric and algebraic topology. 
Prominent among these tools are digital versions 
of continuous functions and homotopy. Digital 
homotopy can be thought 
of as the topology of animated digital images.

In Euclidean topology, finite and infinite spaces 
can have the same homotopy type. E.g., 
the Euclidean line
$\R$ and a point have the same homotopy type.
The analogous statement is not true in 
digital topology; e.g., the digital line $\Z$ 
with the $c_1$ adjacency and a single point do not 
have the same digital homotopy type, despite 
sharing homotopy properties such as having trivial 
fundamental groups.
We introduce in this paper the notions of digital
homotopy similarity, same long homotopy type, and
same real homotopy type, all of which are more general
than digital homotopy equivalence and whose pointed 
versions are less general than having isomorphic
fundamental groups. These notions allow the possibility 
of considering a bounded and an unbounded digital image
as similar with respect to homotopy.

\section{Preliminaries}
We say a connected digital image $X$ has
{\em bounded diameter} if there is
a positive integer $n$ such that
if $x$ and $y$ are members of the
same component of $X$, then
there is a path in $X$ from
$x$ to $y$ of length at most $n$.

Much of the material in this section is quoted or 
paraphrased from other papers in digital topology, 
such as ~\cite{BoSt}.

\subsection{General Properties}
\label{dig-con}
Let $\N$ be the set of natural numbers,
$\N^* = \{0\} \cup \N$, and let
$\Z$ denote the set of integers.
Then $\Z^n$ is the set of lattice points in
Euclidean $n-$dimensional space.

Adjacency relations commonly used for
digital images include the following ~\cite{Kong}.
Two points $p$ and $q$ in $\Z^2$ are $8-adjacent$ if they
are distinct and differ by at most $1$ in each coordinate; $p$ and $q$ in $\Z^2$ are $4-adjacent$ if they are 
8-adjacent and
differ in exactly one coordinate.  Two points $p$ and $q$ in $\Z^3$ are $26-adjacent$ if they
are distinct and differ by at most $1$ in each coordinate; they are $18-adjacent$ if they are 26-adjacent and
differ in at most two coordinates; they are
$6-adjacent$ if they are 18-adjacent and
differ in exactly one coordinate.  For $k \in \{4,8,6,18,26\}$,
a $k-neighbor$ of a lattice point $p$ is a point that is
$k-$adjacent to $p$.

The adjacencies discussed above are generalized as follows.
Let $u, n$ be positive integers, $1 \leq u \leq n$.
Distinct points $p,q \in \Z^n$ are called 
$c_u$-{\em adjacent}, or $c_u$-{\em neighbors}, if
there are at most $u$ distinct coordinates $j$ for which
$|p_j-q_j| \, = \, 1$, and for all other coordinates $j$, $p_j = q_j$.
The notation $c_u$ represents the number of points $q \in \Z^n$
that are adjacent to a given point $p \in \Z^n$ in this sense. Thus the
values mentioned above:
if $n=1$ we have $c_1=2$; if $n=2$ we have $c_1=4$ and $c_2=8$;
if $n=3$ we have $c_1=6$, $c_2=18$, and $c_3=26$.
Yet more general adjacency relations are discussed in~\cite{Herman}.

Let $\kappa$ be an adjacency relation defined on $\Z^n$.
A digital image $X \subset \Z^n$ is $\kappa-connected$~\cite{Herman}
if and only if for
every pair of points $\{x,y\} \subset X$, $x \neq y$, there exists a set
$\{x_0, x_1,\ldots, x_c\} \subset X$ such that $x = x_0$, $x_c = y$,
and $x_j$ and $x_{j+1}$ are $\kappa-$neighbors, $i \in \{0,1,\ldots, c-1\}$.
A {\em $\kappa$-component} of $X$ is a maximal $\kappa$-connected subset of
$X$.

Often, we must assume some adjacency relation for
the {\em white pixels} in $\Z^n$,
{\em i.e.}, the pixels of $\Z^n \setminus X$ (the pixels that
belong to $X$ are sometimes referred to as {\em black pixels}).
In this paper, we are not concerned with adjacencies between
white pixels.

\begin{Definition} {\rm \cite{Boxer94}}
Let $a, b \in \Z$, $a < b$.  A {\rm digital interval}
is a set of the form
\[ [a,b]_{\Z} ~=~ \{z \in \Z ~|~ a \leq z \leq b\} \]
in which $2-$adjacency is assumed. $\Box$
\end{Definition}

\begin{Definition}
\label{dig-cont}
{\rm (\cite{Boxer99}; see also ~\cite{Rosenfeld})}
Let $X \subset \Z^{n_0}$, $Y \subset \Z^{n_1}$.
Let $f:X\rightarrow Y$ be a function.
Let $\kappa_j$ be an adjacency relation defined on 
$\Z^{n_j}$,
$i \in \{0,1\}$.  We say $f$ is
{\rm $(\kappa_0,\kappa_1)-$continuous} if for every
$\kappa_0-$connected subset $A$ of $X$, $f(A)$ is a
$\kappa_1-$connected subset of $Y$. $\Box$
\end{Definition}

See also~\cite{Chen1,Chen2}, where similar
notions are referred to as {\em 
immersions}, {\em gradually varied operators},
and {\em gradually varied mappings}.

We say a function satisfying Definition~\ref{dig-cont}
is {\em digitally continuous}.
This definition implies the following.

\begin{Proposition}
\label{cont-connect}
{\rm (\cite{Boxer99}; see also~\cite{Rosenfeld})}
Let $X$ and $Y$ be digital images.  Then the function
$f: X \rightarrow Y$ is $(\kappa_0,\kappa_1)$-continuous if and only if
for every $\{x_0, x_1\} \subset X$ such that $x_0$ and $x_1$ are
$\kappa_0-$adjacent, either $f(x_0) = f(x_1)$ or $f(x_0)$ and $f(x_1)$
are $\kappa_1-$adjacent. $\Box$
\end{Proposition}

For example, if $\kappa$ is an adjacency relation on a digital image $Y$,
then $f:[a,b]_{\Z} \rightarrow Y$ is $(2,\kappa)-$continuous
if and only if for every
$\{c,c+1\} \subset [a,b]_{\Z}$, either $f(c)=f(c+1)$ or
$f(c)$ and $f(c+1)$ are $\kappa-$adjacent. If some function $f:[0,k]_{\Z} \rightarrow Y$ is $(2,\kappa)-$continuous, we say $f$ is a \emph{$\kappa$-path} from $f(0)$ to $f(k)$ of length $k$. 

We have the following.

\begin{Proposition}
\label{composition}
{\rm \cite{Boxer99}}
Composition preserves digital continuity, {\em i.e.},
if $f:X \rightarrow Y$ and $g:Y \rightarrow Z$ are, respectively,
$(\kappa_0,\kappa_1)-$continuous and $(\kappa_1,\kappa_2)-$continuous
functions, then the composite function $g \circ f: X \rightarrow Z$
is $(\kappa_0,\kappa_2)-$continuous. $\Box$
\end{Proposition}

Digital images $(X,\kappa)$ and $(Y,\lambda)$ are
$(\kappa, \lambda)-isomorphic$ (called
$(\kappa, \lambda)-homeomorphic$ in ~\cite{Boxer94,Boxer05}) if there
is a bijection $h: X \rightarrow Y$ that is $(\kappa,\lambda)$-continuous,
such that the function $h^{-1}: Y \rightarrow X$ is
$(\lambda, \kappa)$-continuous.



\subsection{Digital homotopy}
A homotopy between continuous functions may be thought of as
a continuous deformation of one of the functions into the 
other over a finite time period.

\begin{Definition}{\rm (\cite{Boxer99}; see also \cite{Khalimsky})}
\label{htpy-2nd-def}
Let $X$ and $Y$ be digital images.
Let $f,g: X \rightarrow Y$ be $(\kappa,\kappa')$-continuous functions.
Suppose there is a positive integer $m$ and a function
$F: X \times [0,m]_{\Z} \rightarrow Y$
such that

\begin{itemize}
\item for all $x \in X$, $F(x,0) = f(x)$ and $F(x,m) = g(x)$;
\item for all $x \in X$, the induced function
      $F_x: [0,m]_{\Z} \rightarrow Y$ defined by
          \[ F_x(t) ~=~ F(x,t) \mbox{ for all } t \in [0,m]_{\Z} \]
          is $(2,\kappa')-$continuous. That is, $F_x(t)$ is a path in $Y$.
\item for all $t \in [0,m]_{\Z}$, the induced function
         $F_t: X \rightarrow Y$ defined by
          \[ F_t(x) ~=~ F(x,t) \mbox{ for all } x \in  X \]
          is $(\kappa,\kappa')-$continuous.
\end{itemize}
Then $F$ is a {\rm digital $(\kappa,\kappa')-$homotopy between} $f$ and
$g$, and $f$ and $g$ are {\rm digitally $(\kappa,\kappa')-$homotopic in} $Y$.
If for some $x \in X$ we have $F(x,t)=F(x,0)$ for all
$t \in [0,m]_{\Z}$, we say $F$ {\rm holds $x$ fixed}, and $F$ is a {\rm pointed homotopy}.
$\Box$
\end{Definition}

We indicate a pair of homotopic functions as
described above by $f \simeq_{\kappa,\kappa'} g$.
When the adjacency relations $\kappa$ and $\kappa'$ are understood in context,
we say $f$ and $g$ are {\em digitally homotopic}
to abbreviate ``digitally 
$(\kappa,\kappa')-$homotopic in $Y$," and write
$f \simeq g$.

\begin{Proposition}
\label{htpy-equiv-rel}
{\rm ~\cite{Khalimsky,Boxer99}}
Digital homotopy is an equivalence relation among
digitally continuous functions $f: X \rightarrow Y$.
$\Box$
\end{Proposition}



\begin{Definition}
{\rm ~\cite{Boxer05}}
\label{htpy-type}
Let $f: X \rightarrow Y$ be a $(\kappa,\kappa')$-continuous function and let
$g: Y \rightarrow X$ be a $(\kappa',\kappa)$-continuous function such that
\[ f \circ g \simeq_{\kappa',\kappa'} 1_X \mbox{ and }
   g \circ f \simeq_{\kappa,\kappa} 1_Y. \]
Then we say $X$ and $Y$ have the {\rm same $(\kappa,\kappa')$-homotopy type}
and that $X$ and $Y$ are $(\kappa,\kappa')$-{\rm homotopy equivalent}, denoted 
$X \simeq_{\kappa,\kappa'} Y$ or as
$X \simeq Y$ when $\kappa$ and $\kappa'$ are
understood.
If for some $x_0 \in X$ and $y_0 \in Y$ we have
$f(x_0)=y_0$, $g(y_0)=x_0$,
and there exists a homotopy between $f \circ g$
and $1_X$ that holds $x_0$ fixed, and 
a homotopy between $g \circ f$
and $1_Y$ that holds $y_0$ fixed, we say
$(X,x_0,\kappa)$ and $(Y,y_0,\kappa')$ are
{\rm pointed homotopy equivalent} and that $(X,x_0)$ 
and $(Y,y_0)$ have the 
{\rm same pointed homotopy type}, denoted 
$(X,x_0) \simeq_{\kappa,\kappa'} (Y,y_0)$ or as
$(X,x_0) \simeq (Y,y_0)$ when 
$\kappa$ and $\kappa'$ are understood.
$\Box$
\end{Definition}

It is easily seen, from 
Proposition~\ref{htpy-equiv-rel}, that having the
same homotopy type (respectively, the same
pointed homotopy type) is an equivalence relation
among digital images (respectively, among pointed
digital images).

For $p \in Y$, we denote by $\overline{p}$ the constant
function $\overline{p}:X \rightarrow Y$ defined by $\overline{p}(x) = p$ for
all $x \in X$.

\begin{Definition}
\label{htpy-trivial}
A digital image $(X,\kappa)$ is
$\kappa$-{\rm contractible \cite{Khalimsky,Boxer94}}
if its identity map is $(\kappa, \kappa)$-homotopic to a 
constant function $\overline{p}$ for some $p \in X$.
If the homotopy of the contraction holds $p$ fixed,
we say $(X, p, \kappa)$ is {\rm pointed $\kappa$-contractible}.
$\Box$
\end{Definition}

When $\kappa$ is understood, we speak of {\em contractibility}
for short.
It is easily seen that $X$
is contractible
if and only if $X$
has the homotopy type 
of a one-point digital image.





\subsection{Fundamental group}
Inspired by the fundamental group of a topological space,
several researchers~\cite{Stout,Kong,Boxer99,BoSt}
have developed versions of a fundamental
group for digital images. These are not all equivalent; 
however, it is shown in~\cite{BoSt} that the 
version of the fundamental group developed in that 
paper is equivalent to the version in ~\cite{Boxer99}.

In the following, we present the version of
the digital fundamental group developed in~\cite{BoSt}.

Given a digital image $X$, a continuous function 
$f: \N^* \rightarrow X$ is an {\em eventually
constant path} or {\em EC path} if there is some point 
$c \in X$ and some $N \geq 0$ such that $f(x) = c$ whenever
$x \geq N$. We abbreviate the latter by $f(\infty) = 
c$. The endpoints of an EC path $f$ are
the two points $f(0)$ and $f(\infty)$.
If $f$ is an EC path and $f(0) = f(\infty)$, 
we say $f$ is an {\em EC loop}, and $f(0)$ is 
called the {\em basepoint} of this loop.

We say that a homotopy $H: [0,k]_{\Z} \times \N^* \to X$ 
between EC paths is an 
{\em EC homotopy} when the function $H_s: \N^* \rightarrow X$
defined by $H_s(t) = H(s, t)$ is an EC path for 
all $s \in [0, k]_{\Z}$.
To indicate an EC homotopy, we write
$f \simeq ^{EC} g$, or $f \simeq_{\kappa} ^{EC} g$
if it is desirable to state the adjacency $\kappa$ of $X$. 
We say an EC homotopy $H$
{\em holds the endpoints fixed} when 
$H_t(0) = f(0) = g(0)$ and there is a $c \in \N$ such that 
$n \geq c$ implies $H_t(n) = f(n) = g(n)$ for all $t$.

Given an EC loop $f: \N^* \rightarrow X$, we let
\[ N_f = \min\{m \in \N^* \, | \, n \geq m
         \mbox{ implies } f(n)=f(m)\}.
\]

For $x_0 \in X$, suppose $f_0, f_1: \N^* \rightarrow X$
are $x_0$-based EC loops. Define an $x_0$-based EC loop 
$f_0 * f_1: \N^* \rightarrow X$ via
\[ f_0 * f_1(n)= \left \{ \begin{array}{ll}
                  f_0(n) & \mbox{if } 0 \leq n \leq N_{f_0}; \\
                  f_1(n-N_{f_0}) & \mbox{if } N_{f_0} \leq n.
                  \end{array} \right .
\]

Given an $x_0$-based EC loop 
$f: \N^* \rightarrow X$, we denote by
$[f]_X$, or $[f]$ when $X$ is understood, the
equivalence class of EC loops that are 
homotopic to $f$ in $X$ holding the endpoints fixed. 
We let $\Pi_1^{\kappa}(X,x_0)$ be the set of 
all such sets $[f]$. The $*$ operation enables us to 
define an operation on $\Pi_1^{\kappa}(X,x_0)$ via 
\[ [f] \cdot [g] = [f * g].
\]
This operation is well defined, and makes $\Pi_1^{\kappa}(X,x_0)$
into a group in which the identity element is the class 
$[\overline{x_0}]$ of the constant loop $\overline{x_0}$ and in
which inverse elements are given by $[f]^{-1}=[f^{-1}]$, where
$f^{-1}: {\N^*} \rightarrow X$ is the EC loop defined by
\[ f^{-1}(n) = \left \{ \begin{array}{ll}
               f(N_f - n) & \mbox{if } 0 \leq n \leq N_f; \\
               x_0 & \mbox{if } n \geq N_f.
               \end{array} \right .
\]



\section{Homotopically similar images}
\label{sim-sec}
In Euclidean topology, it is often possible to say that a 
bounded space $X$ and an unbounded space $Y$ 
have the same homotopy type.  For example, a single 
point and $n$-dimensional Euclidean space 
${\R}^n$ have the same homotopy type, 
roughly since the points of ${\R}^n$ can
be moved continuously within ${\R}^n$ over 
a finite time interval to a single point. 
However, Definition~\ref{htpy-2nd-def} does not 
permit a digital image with unbounded diameter
to have the homotopy type of an image with bounded diameter, since the 
second factor of the domain of a homotopy is a 
finite interval $[0,m]_{\Z}$. In this paper, we seek to 
circumvent this limitation. One of the ways we do so depends on the following.

\begin{Definition}
\label{htpy-sim-def}
Let $X$ and $Y$ be digital images.
We say $(X,\kappa)$ and $(Y,\lambda)$
are {\rm homotopically similar}, denoted
$X \simeq_{\kappa,\lambda}^s Y$, if
there exist subsets
$\{X_j\}_{j=1}^{\infty}$ of $X$ and
$\{Y_j\}_{j=1}^{\infty}$ of $Y$ such that:
\begin{itemize}
\item $X = \bigcup_{j=1}^{\infty} X_j$, 
$Y = \bigcup_{j=1}^{\infty} Y_j$, 
and, for all $j$,
$X_j \subset X_{j+1}$,
$Y_j \subset Y_{j+1}$. 
\item There are continuous functions
      $f_j: X_j \rightarrow Y_j$, 
      $g_j: Y_j \rightarrow X_j$ such that
      $g_j \circ f_j \simeq_{\kappa,\kappa} 1_{X_j}$ and
      $f_j \circ g_j \simeq_{\lambda,\lambda} 1_{Y_j}$.
\item For $v \leq w$, 
      $f_w|X_v \simeq_{\kappa,\lambda} f_v$ in $Y_v$ and
      $g_w|Y_v \simeq_{\lambda,\kappa} g_v$ in $X_v$.
\end{itemize}

If all of these homotopies are pointed with respect to
some $x_1 \in X_1$ and $y_1 \in Y_1$,
we say $(X,x_1)$ and $(Y,y_1)$ are 
{\rm pointed homotopically similar},
denoted $(X,x_1) \simeq_{\kappa,\lambda}^s (Y,y_1)$
or $(X,x_1) \simeq^s (Y,y_1)$ when $\kappa$ and
$\lambda$ are understood.
$\Box$
\end{Definition}

\begin{Proposition}
\label{generalize}
If $X \simeq_{\kappa,\lambda} Y$, then 
$X \simeq_{\kappa,\lambda}^s Y$.
If $(X,x_1) \simeq_{\kappa,\lambda} (Y,y_1)$, then 
$(X,x_1) \simeq_{\kappa,\lambda}^s (Y,y_1)$.
\end{Proposition}

{\em Proof}: Let $f: X \rightarrow Y$ and
$g: Y \rightarrow X$ realize a homotopy equivalence
between $(X, \kappa)$ and $(Y,\lambda)$, or a
pointed homotopy equivalence between 
$(X, \kappa, x_1)$ and $(Y,\lambda, y_1)$. Then
corresponding to Definition~\ref{htpy-sim-def},
we can take, for all $j$, $X_j=X$, $Y_j=Y$, 
$f_j=f$, $g_j=g$.
$\Box$

Although Definition~\ref{htpy-sim-def} does not
require it, we often choose the $X_j$ and $Y_j$
to be finite sets.
Example~\ref{contractible} has an image with bounded diameter and 
an image with unbounded diameter that are not homotopy equivalent 
but are pointed homotopically similar.

\begin{Theorem}
\label{equiv-for-bdd}
Let $X$ and $Y$ be finite digital images. Then
$X \simeq_{\kappa,\lambda} Y$ if and only if
$X \simeq_{\kappa,\lambda}^s Y$, and
$(X,x_0)\simeq_{\kappa,\lambda}(Y,y_0)$ 
if and only if 
$(X,x_0)\simeq_{\kappa,\lambda}^s (Y,y_0)$.
\end{Theorem}

{\em Proof}:
That $X \simeq_{\kappa,\lambda} Y$ implies
$X \simeq_{\kappa,\lambda}^s Y$ is shown in
Proposition~\ref{generalize}. To show the converse: if 
$X \simeq_{\kappa,\lambda}^s Y$, let 
$\{X_j,Y_j,f_j,g_j\}_{j=1}^{\infty}$ be as in
Definition~\ref{htpy-sim-def}. Since $X$ and $Y$ are
finite, there exists a
positive integer $m$ such that $i \geq m$ implies
$X=X_i$ and $Y=Y_i$. Then $f_m: X=X_m \rightarrow Y_m=Y$
and $g_m: Y=Y_m \rightarrow X_m=X$ satisfy
\[ g_m \circ f_m \simeq_{\kappa,\kappa} 1_X, ~ 
   f_m \circ g_m \simeq_{\lambda,\lambda} 1_Y. \]
Thus, $X \simeq_{\kappa,\lambda} Y$.

A similar argument
yields the pointed assertion. $\Box$

In Example~\ref{2-infinite}, we show that two 
digital images with unbounded diameters can be homotopically
similar but not homotopy equivalent.

\begin{Theorem}
\label{equiv-rel}
Homotopic similarity and pointed homotopic similarity are
reflexive and symmetric relations among digital images. 
\end{Theorem}

\begin{proof} The assertion follows easily from 
Definition~\ref{htpy-sim-def}. 
\end{proof}

At the current writing, we do not have an answer for the following.

\begin{Question} Is the homotopy
similarity (unpointed or pointed) of digital images
a transitive relation?
\end{Question}

This appears to be a difficult problem. We need a positive resolution to this question if we are to conclude that homotopic similarity is an equivalence relation. Notice that
if $A \simeq^s B$ via subsets
$\{A_i\}_{i=1}^{\infty}$ of $A$ and
$\{B_i\}_{i=1}^{\infty}$ of $B$,
and $B \simeq^s C$ via subsets
$\{B_i'\}_{i=1}^{\infty}$ of $B$
and $\{C_i\}_{i=1}^{\infty}$ of $C$,
we would have $A \simeq^s C$ if 
$B_i = B_i'$ for infinitely many $i$,
but one can easily construct examples for which 
the latter is not satisfied. 

We show transitivity for the 
following special case.

\begin{Theorem}
\label{sim-finite}
Let $B$ be finite.
Let $A \simeq^s B \simeq^s C$.
Then $A \simeq^s C$.
If $(A,a_0) \simeq^s (B,b_0) \simeq^s (C,c_0)$, then $(A,a_0) \simeq^s (C,c_0)$.
\end{Theorem}

\begin{proof} We sketch a proof
for the unpointed assertion. A
similar argument yields the
pointed assertion.

Let $A \simeq^s B$ via
$A = \bigcup_{i=1}^{\infty}A_i$,
$B = \bigcup_{i=1}^{\infty}B_i$,
as in Definition~~\ref{htpy-sim-def}.
Let $B \simeq^s C$ via
$B = \bigcup_{i=1}^{\infty}B_i'$,
$C = \bigcup_{i=1}^{\infty}C_i$,
as in Definition~~\ref{htpy-sim-def}. Since $B$ is finite,
there exists $i_0$ such that
$i \ge i_0$ implies $B_i=B=B_i'$. Since
homotopy of continuous functions and homotopy type of digital images
are transitive relations, it follows
easily from Definition~~\ref{htpy-sim-def} that $A \simeq^s C$.
\end{proof}

\section{Long homotopy type}
In this section, we introduce {\em long homotopy
type}. We obtain for this notion several properties
analogous to those discussed in Section~\ref{sim-sec} for
homotopic similarity. 

The following definition is a step in the direction of the idea
that a long homotopy is a homotopy over an
infinite time interval. The following is essentially an EC version of Definition \ref{htpy-2nd-def}.

\begin{Definition}
\label{long-def}
Let $(X,\kappa)$ and $(Y,\lambda)$ be digital
images. Let $f,g: X \rightarrow Y$ be continuous.
Let $F: X \times {\N}^* \rightarrow Y$ be
a function such that
\begin{itemize}
\item for all $x\in X$, $F(x,0)=f(x)$ and there exists 
$n \in {\N}^*$ such that $t \geq n$ implies
$F(x,t)=g(x)$.
\item For all $x\in X$, the induced function $F_x: \N^* \to Y$ defined by 
\[ F_x(t) = F(x,t) \text{ for all } t\in [0,\infty]_\Z \]
is an EC-path in $Y$. 
\item For all $t\in \N^*$, the induced function $F_t:X \to Y$ defined by
\[ F_t(x) = F(x,t) \text{ for all } x\in X \]
is $(\kappa,\lambda)$--continuous.
\end{itemize}
Then $F$ is an {\rm l-homotopy} from $f$ to $g$.
If for some $x_0 \in X$ and $y_0 \in Y$ we have
$F(x_0,t)=y_0$ for all $t \in \N^*$, we say
$F$ is a {\rm pointed l-homotopy}. We write
$f \simeq_{\kappa,\lambda}^l g$, or 
$f \simeq^l g$ when the adjacencies $\kappa$ and
$\lambda$ are understood, to indicate that
$f$ and $g$ are l-homotopic functions.
$\Box$
\end{Definition}

Note that the definition above generalizes EC homotopy of paths: if two EC paths $f,g: [0,\infty]_\Z \to Y$ are EC homotopic, then the EC homotopy from $f$ to $g$ is an l-homotopy of $f$ to $g$. 

\begin{Proposition}
\label{l-finite-image}
Let $f,g: X \rightarrow Y$ be (unpointed or pointed)
continuous functions between digital images. If $f$ and
$g$ are (unpointed or pointed) homotopic in $Y$, then
$f$ and $g$ are (unpointed or pointed, respectively) l-homotopic in $Y$. The converse is true if $X$ is finite.
\end{Proposition}

{\em Proof}: We give a proof for the unpointed assertions.
The pointed assertions are proven similarly.

If $f \simeq g$, there is a homotopy
$h: X \times [0,m]_{\Z} \rightarrow Y$ such that
$h(x,0)=f(x)$ and $h(x,m)=g(x)$. Then the function
$H: X \times [0,\infty]_{\Z} \rightarrow Y$ defined by $H(x,t)= h(x,\min\{m,t\})$ is an l-homotopy from $f$ to $g$.

Suppose $X$ is finite and $f \simeq^l g$. Then there is
an l-homotopy $H: X \times [0,\infty]_{\Z} \rightarrow Y$
such that $H(x,0)=f(x)$ and, for all $x \in X$, there exists
\[ t_x = \min\{t \in \N^* \, | \, s \geq t 
         \mbox{ implies } H(x,s)=H(x,t) \}.
\]
Let $m=\max\{t_x \, | \, x \in X\}$. Then the function
$h: X \times [0,m]_{\Z} \rightarrow Y$ defined by
$h(x,t)=H(x,t)$ is a homotopy from $f$ to $g$.
$\Box$

At the current writing, we have
not found answers to the following.
\begin{Question}
\label{symm}
(Unpointed and pointed versions:) 
      Is l-homotopy a symmetric relation among
      continuous functions between digital images?
\end{Question}

\begin{Question}
\label{trans}
(Unpointed and pointed versions:) 
      Is l-homotopy a transitive relation among
      continuous functions between digital images?
\end{Question}

Positive answers to 
Questions~\ref{symm} and~\ref{trans} are necessary for
l-homotopy to be an equivalence relation. In the absence of such
results, we proceed with the following
definition.

\begin{Definition}
\label{revised-L}
Let $(X,\kappa)$ and $(Y,\lambda)$ be digital
images. Let $f,g: X \rightarrow Y$ be continuous.
Let $F: X \times \Z \to Y$ be
a function such that
\begin{itemize}
\item for all $x\in X$, there exists 
$N_{F,x} \in \N$ such that $t \leq -N_{F,x}$ implies
$F(x,t)=f(x)$ and $t \geq N_{F,x}$ implies
$F(x,t)=g(x)$.
\item For all $x\in X$, the induced function $F_x:\Z \to Y$ defined by 
\[ F_x(t) = F(x,t) \text{ for all } t\in \Z \]
is $(c_1,\lambda)$-continuous. 
\item For all $t\in \Z$, the induced function $F_t:X \to Y$ defined by
\[ F_t(x) = F(x,t) \text{ for all } x\in X \]
is $(\kappa,\lambda)$--continuous.
\end{itemize}
Then $F$ is a {\rm long homotopy} from $f$ to $g$.
If for some $x_0 \in X$ and $y_0 \in Y$ we have
$F(x_0,t)=y_0$ for all $t \in \N^*$, we say
$F$ is a {\rm pointed long homotopy}. We write
$f \simeq_{\kappa,\lambda}^L g$, or 
$f \simeq^L g$ when the adjacencies $\kappa$ and
$\lambda$ are understood, to indicate that
$f$ and $g$ are long homotopic functions.
$\Box$
\end{Definition}

It is easy to show that the existence of an l-homotopy implies a long homotopy:

\begin{Proposition}
\label{l-implies-long}
Let $f,g: (X,\kappa) \to (Y,\lambda)$
be continuous functions between
digital images. If $f \simeq^l g$
then $f \simeq^L g$. If the l-homotopy
between $f$ and $g$ is pointed, then
the long homotopy between $f$ and $g$ is pointed.
\end{Proposition}

\begin{proof}
Clearly, if $F: X \times \N^* \to Y$
is a (pointed) l-homotopy from
$f$ to $g$, then the function
$F': X \times \Z \to Y$ defined by
\[ F'(x,t) = \left \{ \begin{array}{ll}
   F(x,t) & \mbox{if } t \ge 0; \\
   F(x,0) & \mbox{if } t < 0,
   \end{array} \right .
\]
is a (pointed) long homotopy from
$f$ to $g$.
\end{proof}

The same argument used in Proposition~\ref{l-finite-image} proves the corresponding assertion for long homotopy:
\begin{Proposition}
\label{long-finite-image}
Let $f,g: X \rightarrow Y$ be (unpointed or pointed)
continuous functions between digital images. If $f$ and
$g$ are (unpointed or pointed) homotopic in $Y$, then
$f$ and $g$ are (unpointed or pointed, respectively) long homotopic in $Y$. The converse is true if $X$ is finite.
\end{Proposition}

%

Unlike with l-homotopy, it is easy to see that long homotopy is symmetric.

\begin{Theorem}
\label{long-equiv}
Long homotopy and pointed long homotopy
are reflexive and symmetric relations.
\end{Theorem}

\begin{proof}
We state a proof for the unpointed
assertion. The same argument
works for the pointed assertion.

For the reflexive property, we note
the following.
Given a continuous function
$f: (X,\kappa) \to (Y,\lambda)$, it
is clear that the function
$F: X \times \Z \to Y$ given by
$F(x,t)=f(x)$ is a long homotopy
from $f$ to $f$.

For the symmetric property, we
note that if $F: X \times \Z \to Y$ is
a long homotopy from $f$ to $g$, where
$f,g: (X,\kappa) \to (Y,\lambda)$ are
continuous, then
$F': X \times \Z \to Y$, defined by
$F'(x,t)=F(x,-t)$, is a long homotopy
from $g$ to $f$.
\end{proof}

At the current writing, we lack an answer to the following.

\begin{Question}
Is long homotopy (pointed or unpointed) between continuous
functions a transitive relationship?
\end{Question}

This seems to be a difficult problem. If $X$ is finite
and $f,g,h: X \to Y$ with
$f \simeq^L g \simeq^L h$, then $f  \simeq^L h$ since it 
follows from Proposition~\ref{long-finite-image} that 
in this case, long homotopy coincides with homotopy, which
is transitive. 
To demonstrate the difficulty involved in the general case,
we will prove transitivity for another special case. 

It is easy to see that if $c,d:X\to Y$ are two different constant maps whose constant values $c=c(x)$ and $d=d(x)$ are in the same component of $X$, then $c$ and $d$ are homotopic, and thus long homotopic. Thus the following theorem is a very special case of transitivity, but the proof is already somewhat involved.

Say that a digital image $X$ is \emph{locally finite} when each point $x\in X$ is adjacent to only finitely many other points of $X$.
E.g., if $X$ is finite, or if $X$
has a $c_u$-adjacency, then $X$ is locally finite.

\begin{Theorem}
\label{transitive-thru-consts}
Let $X$ be locally finite, and let $f:X\to Y$ be a continuous function, and let $c,d:X\to Y$ be two constant functions with constant values $c$ and $d$ in the same component of $Y$. If $f\simeq^L c$, then $f\simeq^L d$.
\end{Theorem}
\begin{proof}
Let $\sigma:[0,k]_\Z \to Y$ be a path from $c$ to $d$. Our proof is by induction on $k$. If $k=0$, then $c=d$ and there is nothing to prove. Letting $c' = \sigma(k-1)$, for our induction case we may assume that $f \simeq^L c'$, and we will show that $f\simeq^L d$. (Note that $c'$ and $d$ are adjacent.)

Let $H:X\times \Z \to Y$ be a long homotopy of $f$ to $c'$. Then for each $x \in X$, there is a number $N_x$ such that, whenever $t \ge N_x$, we have $H(x,t) = c'$. Since $X$ is locally finite, there is a number $N'_x \ge N_x$ such that, whenever $t \ge N'_x$, we have $H(x',t) = c'$  for every $x'$ adjacent to $x$. 

Then we define $G:X\times \Z \to Y$ as:
\[ G(x,t)= \begin{cases} H(x,t) & \text{ if } t \le N'_x; \\
d & \text{ if } t > N'_x. \end{cases} \]

We claim that $G$ is a long homotopy of $f$ to $d$.
It is clear that for all $x \in X$
there exists $n_x \in \N$ such that $t \le -n_x$ implies $G(x,t) = H(x,t) = f(x)$ and
$t \ge n_x$ implies $G(x,t)= d$.
Furthermore, the induced function $G_x(t)$ is given by:
\[ G_x(t) = \begin{cases} H_x(t) & \text { if } t \le N'_x; \\
d &\text { if } t > N'_x. \end{cases}, \]
which is continuous since $H(x,N'_x) = c'$ is adjacent to $d$.

Lastly we show that the induced function $G_t(x)$ is continuous: take any point $y$ adjacent to $x$, and we will show that $G_t(x)$ is adjacent or equal to $G_t(y)$.
\begin{itemize}
\item
When $t \le N'_x$, we have $G_t(x) = H_t(x)$,
which is adjacent or equal to $H_t(y)$
because $H$ is a homotopy.
\begin{itemize}
\item If $H_t(y)=G_t(y)$, we have the
desired conclusion that $G_t(x)$ is adjacent or equal to $G_t(y)$.
\item Otherwise, $G_t(y)=d \neq H_t(y)$, so $t > N_y'$, which implies
$H(x,t)=c'$. Thus, $G_t(x) \in \{c',d\}$, so
$G_t(x)$ is adjacent or equal to $G_t(y)$.
\end{itemize}
\item
If $t > N'_x$ then $G_t(x) = d$. For $G_t(y)$ there are two cases. 
If $t \ge N'_y$ then $G_t(y) = d = G_t(x)$. If $t < N'_y$ we still must have $t \ge N_y$ since $t > N'_x$ and $y$ is adjacent to $x$. Thus in this case $G_t(y) = H_t(y) = c'$, and thus $G_t(x) = d$ and $G_t(y) = c'$ are adjacent as desired.
\end{itemize}
\end{proof}

\begin{Definition}
\label{long-equiv-def}
Let $f: (X, \kappa) \rightarrow (Y,\lambda)$ and
$g: (Y, \lambda) \rightarrow (X,\kappa)$ be
continuous functions. Suppose
$g \circ f \simeq^L 1_X$ and
$f \circ g \simeq^L 1_Y$. Then we say
$(X,\kappa)$ and $(Y,\lambda)$  
have the same {\rm long homotopy type}, denoted 
$X\simeq_{\kappa,\lambda}^L Y$ or simply $X \simeq^L Y$.
If there exist $x_0 \in X$ and $y_0 \in Y$ such
that $f(x_0)=y_0$, $g(y_0)=x_0$,
the long homotopy
$g \circ f \simeq^L 1_X$
holds $x_0$ fixed, and
the long homotopy
$f \circ g \simeq^L 1_Y$
holds $y_0$ fixed, then 
$(X,x_0,\kappa)$ and $(Y,y_0,\lambda)$ have the 
same {\rm pointed long homotopy type}, denoted
$(X,x_0) \simeq_{\kappa,\lambda}^L (Y,y_0)$ or
$(X,x_0) \simeq^L (Y,y_0)$.
$\Box$
\end{Definition}

\begin{Proposition}
\label{hequiv-implies-long}
If $X \simeq_{\kappa,\lambda}Y$, then
$X \simeq_{\kappa,\lambda}^L Y$.
If $(X,x_0) \simeq_{\kappa,\lambda} (Y,y_0)$, then $(X,x_0) \simeq_{\kappa,\lambda}^L (Y,y_0)$. The converses of both statements hold when $X$ and $Y$ are both finite.
\end{Proposition}

{\em Proof}: The assertions follow from Definition~\ref{long-equiv-def} and Proposition~\ref{long-finite-image}.
$\Box$

%
%
%
%
%
%

\begin{Theorem}
\label{long-reflex-symm}
Long homotopy type, and pointed long
homotopy type, are reflexive and
symmetric relations among digital
images.
\end{Theorem}

\begin{proof}
The assertions follow easily from
Definition~\ref{long-equiv-def}.
\end{proof}

At the current writing, we lack an answer to the following.

\begin{Question}
Is long homotopy type (unpointed or pointed)
a transitive relation among digital images?
\end{Question}

This appears to be a difficult
problem. A positive resolution to
this question is necessary in order for us to conclude that long homotopy type is an equivalence relation. Since homotopy equivalence is an equivalence relation,
Proposition~\ref{hequiv-implies-long} implies
(for both the pointed and unpointed
questions) that if there exists
an example of non-transitivity for
long homotopy type, i.e., images
$A,B,C$ such that
$A \simeq^L B \simeq^L C$ with
$A$ and $C$ not long homotopically equivalent, then at least one of
$A,B,C$ must be infinite.

In the next result,
we prove transitivity for a special case. The following resembles Theorem \ref{sim-finite} for homotopic similarity, but requires that the intermediate image be a single point. It does not seem easy to generalize to finite sets as in Theorem \ref{sim-finite}.

\begin{Theorem}
\label{long-nullhtpc}
Let $X \simeq^L \{a\} \simeq^L Y$.
Then $X \simeq^L Y$.
If $(X,x_0) \simeq^L (\{a\},a) \simeq^L (Y,y_0)$,
then $(X,x_0) \simeq^L (Y,y_0)$.
\end{Theorem}

\begin{proof} 
We state a proof for the pointed
assertion; the unpointed assertion is handled
similarly.

By hypothesis, there are pointed functions
$f: (X,x_0) \to (\{a\},a)$ and
$g: (\{a\},a) \to (X,x_0)$ and a pointed long
homotopy $H: X \times \Z \to X$ from
$1_X$ to $g \circ f=\overline{x_0}$. Similarly, there are pointed functions
$h: (Y,y_0) \to (\{a\},a)$ and
$k: (\{a\},a) \to (Y,y_0)$ and a pointed long
homotopy $K: Y \times \Z \to Y$ from
$1_Y$ to $k \circ h=\overline{y_0}$.

Let $\overline{y_0}'$ be the constant
function from $X$ to $Y$.
Let $\overline{x_0}'$ be the constant
function from $Y$ to $X$. Then
$H$ is a pointed long homotopy from $1_X$ to 
$\overline{x_0}=\overline{x_0}' \circ \overline{y_0}'$,
and $K$ is a pointed long homotopy from $1_Y$ to 
$\overline{y_0}=\overline{y_0}' \circ \overline{x_0}'$.
The assertion follows.
\end{proof}

A digital image with bounded diameter and an image with
infinite diameter
cannot have the same homotopy type, but
Example~\ref{contractible} shows that such a pair of images can have the same long homotopy 
type. Example~\ref{2-infinite} 
gives two digital images with
unbounded diameters that have the same long homotopy 
type but not the same homotopy type.

\section{Real homotopy type}
\label{real-homotopy}
In this section we present another generalization of digital homotopy that we call {\em real homotopy}.
As in the case of long homotopy, we will allow the time interval to be infinite, this time using the real interval $[0,1]$. Though nondiscrete sets are not typically used in digital topology, we will see as in the other sections that real homotopy and digital homotopy are equivalent when the images under consideration are finite. The advantage in using the real interval rather than the integer interval $[0,\infty)$ is that two copies of $[0,1]$ can be concatenated in a natural way, which allows us to prove that real homotopy is transitive.

It also turns out that long homotopy can tell us a lot about
real homotopy.

We begin with a preliminary definition that is a kind of continuity property for a function from a real interval into a digital image. Informally we want to require that such a function be locally constant with jump discontinuities only between adjacent points.
\begin{Definition}
\label{real-path}
Let $(X,\kappa)$ be a digital image, and $[0,1] \subset \R$ be the unit interval. A function $f:[0,1] \to X$ is a \emph{real [digital] [$\kappa$-]path in $X$} if:
\begin{itemize}
\item there exists $\epsilon_0>0$ such that $f$ is constant on $(0,\epsilon_0)$ with constant value equal or  $\kappa$-adjacent to $f(0)$,
\item there exists $\epsilon_1>0$ such that $f$ is constant on $(1-\epsilon_1, 1)$ with constant value equal or  $\kappa$-adjacent to $f(1)$,
\item for each $t\in (0,1)$ there exists $\epsilon_t >0$ such that $f$ is constant on each of the intervals $(t-\epsilon_t,t)$ and $(t,t+\epsilon_t)$, and these two constant values are equal or $\kappa$-adjacent, with at least one of them equal to $f(t)$.
\end{itemize}
If $t=0$ and $f(0) \neq f((0,\epsilon_0))$, or $0<t<1$ and
the two constant 
values $f((t-\epsilon_t,t))$ and $f((t,t+\epsilon_t))$
are not equal, or $t=1$ and $f(1) \neq f((1-\epsilon_1,1))$, we say $t$ is a \emph{jump} of $f$.
\end{Definition}



In fact such a real path can only have finitely many jumps, as the following Proposition shows.

\begin{Proposition}
\label{finite-jumps}
Let $p, q \in (X, \kappa)$. 
Let $f : [a, b] \rightarrow X$ be a real $\kappa$-path
from $p$ to $q$. Then the number of jumps of $f$ is 
finite.
\end{Proposition}

\begin{proof}
Suppose $f$ has an infinite set of jumps in the domain
$[a, b]$. By the Bolzano-Weierstrass
Theorem this set of jumps has an accumulation point. Thus there exists $t_0 \in [a, b]$ and a sequence
of distinct jumps $\{t_j\}_{j=1}^{\infty}$ such that
$\lim_{j \rightarrow \infty} t_j = t_0$. Then
for every $\varepsilon > 0$, at least one of the
intervals $(t_0 - \varepsilon, t_0)$ and
$(t_0, t_0 + \varepsilon)$ has infinitely many
members of $\{t_j\}_{j=1}^{\infty}$,
contrary to the requirement of 
Definition~\ref{real-path} that there be
$\varepsilon > 0$ such that $f$ is constant on each
of the intervals $(t_0 - \varepsilon, t_0)$ and
$(t_0, t_0 + \varepsilon)$.
\end{proof}

Now we can define real digital homotopy of functions. The following is a ``real'' version of Definitions \ref{htpy-2nd-def} and \ref{long-def}.
\begin{Definition}\label{real-def}
Let $(X,\kappa)$ and $(Y,\kappa')$ be digital images, and let $f,g: X \to Y$ be $(\kappa,\kappa')$ continuous. Then a \emph{real [digital] homotopy} of $f$ and $g$ is a function $F:X \times [0,1] \to Y$ such that:
\begin{itemize}
\item for all $x\in X$, $F(x,0) = f(x)$ and $F(x,1) = g(x)$
\item for all $x \in X$, the induced function $F_x:[0,1]\to Y$ defined by
\[ F_x(t) = F(x,t)\, \text{for all $t \in [0,1]$} \]
is a real $\kappa$-path in $X$.
\item for all $t\in [0,1]$, the induced function $F_t:X\to Y$ defined by 
\[ F_t(x) = F(x,t) \, \text{for all $x\in X$} \]
is $(\kappa,\kappa')$--continuous.
\end{itemize}
If such a function exists we say $f$ and $g$ are
{\em real homotopic} and write $f\simeq^\R g$.
If there are points $x_0 \in X$ and $y_0 \in Y$ such
that $F(x_0,t)=y_0$
for all $t \in [0,1]$, we say $f$ and $g$ are
{\em pointed real homotopic}.
\end{Definition}

Unlike long homotopy, real homotopy is easily shown to be an equivalence relation.

\begin{Theorem}
\label{real-equiv}
Real homotopy and pointed real homotopy are 
equivalence relations among continuous functions 
between digital images.
\end{Theorem}
\begin{proof} We give the proof for the unpointed
assertion. A similar argument can be used to establish
the pointed assertion.

Reflexivity is clear: for any digitally continuous function $f:X \to Y$, the function $F(x,t) = f(x)$ is a real homotopy from $f$ to $f$.

For symmetry, let $f,g:X\to Y$ be digitally continuous with $f\simeq^\R g$, and let $F:X\times[0,1] \to Y$ be a real homotopy from $f$ to $g$. Then define $G:X\times [0,1] \to Y$ by $G(x,t) = F(x,1-t)$. It is easy to verify that $G$ is a real homotopy from $g$ to $f$, and so $g\simeq^\R f$. 

For transitivity, let $f,g,h:X\to Y$ with $f\simeq^\R g$ and $g\simeq^\R h$. Let $F,G:X\times[0,1] \to Y$ be homotopies from $f$ to $g$ and $g$ to $h$, respectively. Then define $H:X\times [0,1] \to Y$ by
\[ H(x,t) = \begin{cases} 
F(x,2t) &\text{ if } t \le 1/2, \\
G(x,2t-1) &\text{ if } t \ge 1/2.
\end{cases} \]
Again it is routine to check that $H$ is a real homotopy from $f$ to $h$, and so $f \simeq^\R h$ as desired.
\end{proof}

Next we show that long homotopy of functions implies real homotopy.

\begin{Theorem}
\label{long-implies-real}
Let $X$ and $Y$ be digital images, and let $f,g:X \to Y$ be continuous. If $f\simeq^L g$ then $f\simeq^\R g$. If $f$ and $g$ are pointed long homotopic, then
they are pointed real homotopic.
\end{Theorem}
\begin{proof}
Let $h:(0,1) \to \R$ be a homeomorphism with $\lim_{x\to 0} h(x) = -\infty$ and $\lim_{x\to 1} h(x) = \infty$. For example $h$ can be taken to be a rescaled version of the tangent function. Let $H:X\times \Z \to Y$ be a long homotopy from $f$ to $g$, and define $F:X\times [0,1] \to Y$ by: 
\[
F(x,t) = \begin{cases}
f(x) &\text{ if } t=0, \\
H(x,\lfloor h(t)\rfloor ) &\text{ if } t\in (0,1), \\
g(x) &\text{ if } t=1.
\end{cases}
\]
We claim that $F$ is a real homotopy from $f$ to $g$. We have defined $F$ so that $F(x,0) = f(x)$ and $F(x,1) = g(x)$ for all $x$. Observe also that each induced function $F_t(x)$ is continuous - for $t\in \{0,1\}$ this is true because $f$ and $g$ are continuous, and for other $t$ because $H_s$ is continuous for any $s \in \Z$. 

It remains to show that the induced function $F_x:[0,1] \to Y$ is a real path for every $x$. Note that the value of $H_x(t) = H(x,\lfloor h(t) \rfloor)$ changes only when $h(t)$ is an integer. When $h(t)$ is an integer, the value of $H_x$ can only change from one point of $Y$ to an adjacent point. Thus for any $t \in (0,1)$, there is some $\epsilon_t$ such that $H_x$ is constant on $(t-\epsilon_t,t)$ and $(t,t+\epsilon_t)$, and these constant values are adjacent, and one of them equals $F_x(t)$. 

It remains to show the existence of $\epsilon_0$ and $\epsilon_1$ as in Definition \ref{real-def}.
Because $H$ is a long homotopy, there is a natural number $N_x$ such that $H(x,t) = g(x)$ whenever $t>N_x$ and $H(x,t) = f(x)$ whenever $t<-N_x$. Then choose $\epsilon_0$ with $0<\epsilon_0 < h^{-1}(-N_x)$, and then $F_x$ will be constant on $[0,\epsilon_0)$ as required. Choose $\epsilon_1$ such that $h^{-1}(N_x) < \epsilon_1 < 1$, and $F_x$ will be constant on $(1-\epsilon_1, 1]$, as required. Thus $F_x$ is a real path and so $F$ is a real homotopy.
\end{proof}

Next we show that real homotopy is weaker than digital homotopy, and that the two notions are equivalent when the domain is finite. This result is analogous to Proposition \ref{long-finite-image} for long homotopy. 

\begin{Theorem}
\label{long-and-real}
Let $(X,\kappa)$ and $(Y,\kappa')$ be digital images, 
and let $f,g:X \to Y$ be $(\kappa,\kappa')$--continuous. Then $f\simeq g$ implies $f\simeq^\R g$, and the converse is true when $X$ is finite. If $f$ and $g$ are pointed
homotopic, then they are pointed real homotopic, and the 
converse is true when $X$ is finite.
\end{Theorem}
\begin{proof} We give the proof for the unpointed
assertion. A similar argument can be used to establish
the pointed assertion.

First we assume that $f\simeq g$. This implies, by
Proposition~\ref{l-finite-image}, that
$f \simeq^L g$. Hence by Theorem~\ref{long-implies-real},
$f \simeq^\R g$.

%
%

Now for the converse we assume that $f\simeq^\R g$ and $X$ is finite, and show that $f\simeq g$. Let $F$ be a real homotopy from $f$ to $g$. Since $X$ is finite and each real path $F_x(t)$ has finitely many jumps, there are only finitely many values of $t \in [0,1]$ which can be jumps for any of the paths $F_x(t)$. Let $j_0 < \dots < j_k$ be these jump points, and choose $t_1,\dots, t_{k-1}$ so that $j_i < t_i < j_{i+1}$ for each $i$. Also let $t_0 = 0$ and $t_{k+1} = 1$. Since jumps in real paths only move to adjacent points, $F_x(t_i)$ is adjacent or equal to $F_x(t_{i+1})$ for each $i$. 

Now define $G:X\times [0,k+1]_\Z \to Y$ by $G(x,i) = F(x,t_i)$. We will show that $G$ is a homotopy from $f$ to $g$. We have $G(x,0) = F(x,t_0) = F(x,0) = f(x)$ and $G(x,k+1) = F(x,t_{k+1}) = F(x,1) = g(x)$. Since $F_{t_i}(x)$ is $(\kappa,\kappa')$--continuous for each $t_i\in [0,1]$, we have $G_i(x)$ $(\kappa,\kappa')$-continuous for each $i \in [0,k+1]_\Z$ as required. 

It remains to show that the induced function $G_x:[0,k+1]_\Z \to Y$ is $(2,\kappa')$--continuous for each $x\in X$. Equivalently, we must show that $G_x(i)$ is $\kappa'$--adjacent or equal to $G_x(i+1)$ for each $i$. But we have already stated that $F_x(t_i)=G_x(i)$ is adjacent or equal to $F_x(t_{i+1})=G_x(i+1)$ for each $i$, as desired.
\end{proof}

In the case when $X$ is finite, the converse above shows that a real homotopy implies an ordinary (finite) homotopy, which by Proposition \ref{long-finite-image} implies a long homotopy. Combining all these results gives:
\begin{Corollary}\label{finite-all-equiv}
Let $X$ be a finite digital image, $Y$ be any digital image, and $f,g:X\to Y$ be continuous. Then the following three statements are equivalent: $f\simeq g$, $f\simeq^L g$, and $f\simeq^\R g$. Similar equivalences hold for pointed relations.
\end{Corollary}

The following simple question seems hard to answer:
\begin{Question}
Is Corollary \ref{finite-all-equiv} true without the finiteness assumption?
\end{Question}

With our real homotopy relation we can make the obvious definition for real homotopy type of digital images.

\begin{Definition}
\label{real-htpy-equiv}
We say digital images $(X,\kappa)$ and $(Y, \kappa')$
have the {\em same real homotopy type}, denoted
$X \simeq_{\kappa,\kappa'}^\R Y$ or $X \simeq^\R Y$ when
$\kappa$ and $\kappa'$ are understood,
if there are continuous functions $f: X \rightarrow Y$
and $g: Y \rightarrow X$ such that 
$g \circ f \simeq^\R 1_X$ and $f \circ g \simeq^\R 1_Y$.
If there exist $x_0 \in X$ and $y_0 \in Y$ such that
$f(x_0)=y_0$, $g(y_0)=x_0$, and the real homotopies
implicit above are pointed with respect to $x_0$ and $y_0$, we say $X$ and $Y$ have the 
{\em same pointed real homotopy type}, denoted
$(X,x_0) \simeq_{\kappa,\kappa'}^\R (Y,y_0)$ or
$(X,x_0) \simeq^\R (Y,y_0)$.
\end{Definition}

Having the same real homotopy type is also easily seen to be an equivalence relation.

\begin{Theorem}
\label{real-equiv-rel}
Having the same real homotopy type or pointed real homotopy type is an equivalence relation
among digital images. 
\end{Theorem}

\begin{proof}
We prove the unpointed assertion. Simple modifications to the argument give the pointed assertion. 

For the reflexive property it is easy to see that an identity map
$1_X$ shows that $X \simeq^\R X$. The symmetric property follows from
the symmetry in Definition~\ref{real-htpy-equiv}. It remains to prove transitivity.

Suppose
$X \simeq^\R Y \simeq^\R W$. Then there are continuous
functions $f: X \rightarrow Y$, $f': Y \rightarrow X$,
$g: Y \rightarrow W$, and $g': W \rightarrow Y$, and
real homotopies $F: X \times [0,1] \rightarrow X$ from
$f' \circ f$ to $1_X$, 
$F': Y \times [0,1] \rightarrow Y$ from $f \circ f'$ to
$1_Y$, $G: Y \times [0,1] \rightarrow Y$ from
$g' \circ g$ to $1_Y$, and 
$G': W \times [0,1] \rightarrow W$ from $g \circ g'$ 
to $1_W$. We will show that $X \simeq^\R W$ using the functions $g\circ f: X \to W$ and $f'\circ g': W \to X$.

Consider the function 
$H: X \times [0,1] \rightarrow X$
defined by $H(x,t)=f'(G(f(x),t))$. We will show that $H$ is a real homotopy from $f'\circ g'\circ g\circ f$ to $f'\circ f$. 
First observe that 
\[ H(x,0) = f'(G(f(x),0)) = f'(g'(g(f(x)))) = f'\circ g'\circ g\circ f (x)\]
and 
\[ H(x,1) = f'(G(f(x),1)) = f'(f(x)) = f'\circ f(x). \]
Also observe that $H_t = f'\circ G_t\circ f$, and thus $H_t$ is continuous by Theorem \ref{composition}. For $H_x$, we have $H_x = f'\circ G_{f(x)}$. Since $G_{f(x)}$ is a real path and $f'$ is continuous, it is easy to see that $H_x$ is a real path. 

Thus we have shown that $f'\circ g'\circ g \circ f \simeq^\R f'\circ f$. By our assumption we have $f' \circ f \simeq^\R 1_X$, and thus by transitivity of real homotopy (Theorem \ref{real-equiv}) we have $f'\circ g'\circ g \circ f \simeq^\R 1_X$. A similar argument shows that $g\circ f \circ f'\circ g' \simeq 1_W$, and thus $X\simeq^\R W$ as desired.
\end{proof}

Because of the theorem above, when $X\simeq^\R Y$, we say $X$ and $Y$ are \emph{real homotopy equivalent}.

\begin{Corollary}
\label{htpy-type-implies-real-type}
If $X \simeq Y$ then $X \simeq^{\R} Y$. If $(X,x_0) \simeq (Y,y_0)$
then $(X,x_0) \simeq^{\R} (Y,y_0)$. If both $X$ and $Y$ are finite, then the converses hold.
\end{Corollary}

\begin{proof} This follows easily from Theorem~\ref{long-and-real}.
\end{proof}

\begin{Corollary}
\label{long-implies-real-type}
If $X \simeq^L Y$, then 
$X \simeq^\R Y$.
If $(X,x_0) \simeq^L (Y,y_0)$, 
then $(X,x_0) \simeq^\R (Y,y_0)$.
\end{Corollary}

\begin{proof}
This follows from 
Definitions~\ref{long-equiv-def} and 
~\ref{real-htpy-equiv} and 
Theorem~\ref{long-implies-real}.
\end{proof}

Note that in light of 
Corollary~\ref{long-implies-real-type},
Example~\ref{2-infinite} below shows it is possible
for two digital images to be pointed
real homotopy equivalent without being
pointed homotopy equivalent.

\section{Examples}
A finite set and an infinite set cannot have the same digital homotopy type. In particular the set $\Z^n$ is not the same digital homotopy type as a single point, even though the continuous objects they represent ($\R^n$ and a point) are classically homotopy equivalent. In the following example we show that $\Z^n$ is homotopically similar to a point, and also of the same long homotopy type (and thus the same real homotopy type). 
\begin{Example}
\label{contractible}
Let $x=(x_0,x_1,\ldots,x_{n-1}) \in \Z^n$ and 
let $u,v \in [1,n]_{\Z}$.
Then $(\{x\},x)$ and $(\Z^n,x)$ are
$(c_u,c_v)$-pointed homotopically similar and have the same
$(c_u,c_v)$-pointed long homotopy type.
\end{Example}

{\em Proof}:
Corresponding to Definition~\ref{htpy-sim-def},
let $X_j = \{x\}$, let 
$Y_j = \Pi_{i=0}^{n-1} [x_i-j,x_i+j]_{\Z}$,
let $f_j$ be the inclusion map, and let $g_j$ be the
constant map with image $\{x\}$. Since a digital 
cube is pointed $(c_u,c_v)$-contractible with 
respect to any of its points~\cite{Boxer94}, the 
assertion of pointed homotopic similarity follows.

Let $H: \Z^n \times \N^*
     \rightarrow \Z^n$ be the map defined as follows.
Let $H(y,0)=y$.
     For $t>0$, let $q$ be the reduction of $t$ modulo $n$, and we bring the $q$-th coordinate of $H(y,t)$ one unit closer to the $q$-th coordinate of $x$,
     i.e.,
     if $H((y_0,\ldots,y_{n-1}),t-1) = (z_0, \ldots, z_{n-1})$
     then 
     \[ H((y_0,\ldots,y_{n-1}),t) = \left \{ \begin{array}{ll}
     (z_0, \ldots, z_{q-1}, z_q -1, z_{q+1}, \ldots, z_{n-1})
     & \mbox{if } z_q > x_q; \\
     (z_0, \ldots, z_{n-1}) & \mbox{if } z_q = x_q; \\
      (z_0, \ldots, z_{q-1}, z_q +1, z_{q+1}, \ldots, z_{n-1})
     & \mbox{if } z_q < x_q.
     \end{array} \right .
\]
Since, in a single time step, $H$ changes only one coordinate,
it is easily seen that this map holds $x$ fixed and,
for all $u,v \in [1,n]_{\Z}$, is a $(c_u,c_v)$-l-homotopy
from $1_{{\Z}^n}$ to the map $i_x \circ \overline{x}$, where
$\overline{x}: {\Z}^n \to \{x\}$ is a constant map and
$i_x: \{x\} \to \Z^n$ is the inclusion function.
From Proposition~\ref{l-implies-long},
we have that $1_{\Z^n}$ and
$i_x \circ \overline{x}$ are pointed long homotopic. Since
$1_{\{x\}}=\overline{x} \circ i_x$,
the assertion of the same pointed long homotopy type follows. $\Box$

Recall that a tree is an acyclic graph in which
every pair of points is connected by a unique 
injective path. We consider both finite and infinite
trees.

\begin{Example}
\label{tree-ex}
Let $(X,\kappa)$ be a digital image that is a tree.
Let $x_0 \in X$. Then $(X,x_0)$ is pointed homotopically
similar to, and has the same pointed long homotopy type,
as $(\{x_0\},x_0)$.
\end{Example}

\begin{proof}
In the following, for $x \neq x_0$ we use 
$parent(x)$ to denote the unique vertex in
$X$ adjacent to $x$ along the unique path in $X$ from
$x$ to $x_0$; and, by $dist(x,y)$ (the {\em distance} between
vertices $x$ and $y$) we mean the length of the path in $X$
from $x$ to $y$. Note if $dist(x,x_0)=n>0$, then
$dist(parent(x),x_0)=n-1$.

Corresponding to the notation of
Definition~\ref{htpy-sim-def},
let $X_j = \{x \in X \, | \, dist(x,x_0) \leq j\}$. Let
$Y=Y_j=\{x_0\}$. Let $f_j: X_j \to Y_j$ be the function
$f_j(x)=x_0$. Let $g_j: Y_j \to X_j$ be the function
$g(x_0)=x_0$. 
Let $H_j: X_j \times [0,j]_{\Z} \to X_j$ be the
function
\[ H_j(x,t) = \left \{ \begin{array}{ll}
            x & \mbox{if } t=0; \\
            x_0 & \mbox{if } t>0 \mbox{ and }
                  H(x,t-1)=x_0; \\
            parent(H(x,t-1)) & \mbox{if } t>0 
                 \mbox{ and } H(x,t-1) \neq x_0.
            \end{array} \right .
\]
Then $H_j$ is a pointed homotopy from $1_{X_j}$ to
$g_j \circ f_j$. Further, $f_j \circ g_j = 1_{Y_j}$.
It follows easily that $(X,x_0)$ and $(\{x_0\},x_0)$ are
pointed homotopically similar.

Let $H: X \times {\N}^* \to X$ be the
function
\[ H(x,t) = \left \{ \begin{array}{ll}
            x & \mbox{if } t=0; \\
            x_0 & \mbox{if } t>0 \mbox{ and }
                  H(x,t-1)=x_0; \\
            parent(H(x,t-1)) & \mbox{if } t>0 
                 \mbox{ and } H(x,t-1) \neq x_0.
            \end{array} \right .
\]
Then $H$ is a pointed l-homotopy from $1_X$ to
$\overline{x_0}$. From
Proposition~\ref{l-implies-long},
$1_X$ and $i_{\{x_0\}} \circ \overline{x_0}$
are pointed long homotopic, where
$i_{\{x_0\}}$ is the inclusion function of $\{x_0\}$ into $\Z^n$.
Since $\overline{x_0} \circ i_{\{x_0\}}=1_{\{x_0\}}$, it follows that $(X,x_0)$
and $(\{x_0\},x_0)$ have the same pointed long 
homotopy type.
\end{proof}


\begin{Example}
\label{2-infinite}
Let $X = \Z \times \{0\} \subset \Z^2$ and let
$Y = X \cup (\{0\} \times \N) \subset 
\Z^2$ (see Figure~\ref{perp-fig}). Then
$(X,(0,0)) \simeq_{c_1,c_1}^s (Y,(0,0))$, and 
$(X,(0,0)) \simeq_{c_1,c_1}^L (Y,(0,0))$, but 
$(X,(0,0))$ and $(Y,(0,0))$ do not have the same
$(c_1,c_1)$-pointed homotopy type.
\end{Example}

{\em Proof}:
It follows from Example~\ref{tree-ex} and
Theorems~\ref{sim-finite} and~\ref{long-nullhtpc} that 
$(X,(0,0)) \simeq_{c_1,c_1}^s (Y,(0,0))$, and 
$(X,(0,0)) \simeq_{c_1,c_1}^L (Y,(0,0))$.

Suppose $(X,(0,0)) \simeq_{c_1,c_1} (Y,(0,0))$. 
Then there are pointed continuous functions $f: X \rightarrow Y$,
$g: Y \rightarrow X$, and pointed homotopies
$H: X \times [0,k]_{\Z} \rightarrow X$ and
$H': Y \times [0,m]_{\Z} \rightarrow Y$ such that
$H(x,0)=x$ and $H(x,k) = g \circ f(x)$ for all $x \in X$, and
$H'(y,0)=y$ and $H'(y,m) = f \circ g(y)$ for all $y \in Y$.

\begin{figure}
\includegraphics[height=2in]{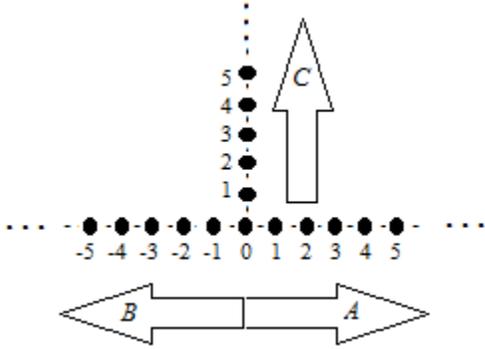}
\caption{The image $Y$ of
Example~\ref{2-infinite}, with the subsets
$A,B,C$ of its proof marked by arrows\label{perp-fig}}
\end{figure}

We show $f$ is a surjection. Note if
$A = \N^* \times \{0\}$, 
$B= \{(n,0) \, | \, n \in \Z, \, (-n,0) \in A\}$,
    and
$C = \{0\} \times \N^*$, we have
$Y = A \cup B \cup C$.
Suppose there exists $p=(u,v) \in Y \setminus f(X)$.
\begin{itemize}
\item If $p \in A$ then, since $f(X)$ is 
      $c_1$-connected and contains $(0,0)$,  
      $A_1=\{(x,0) \, | \, x \geq u \} \subset Y \setminus f(X)$. 
      In particular,
      \begin{equation}
      \label{A2-eq}
A_2=[u,u+2m]_{\Z} \times \{0\} \subset Y \setminus f(X).
      \end{equation}
Since $u>0$, $A_2$ is the set of all points in $Y$ within 
$m$ steps of $(u+m,0)$, so we have
a contradiction of the assumption that $1_Y$ and
$f \circ g$ are homotopic in $m$ steps, as
statement~(\ref{A2-eq}) implies no
point of $(f \circ g)(X)$ is within $m$ steps of
$(u+m,0)$.
      Therefore, we cannot have $p \in A$.
\item The cases $p \in B$ and $p \in C$ yield
      contradictions similarly.
\end{itemize}
Thus, we must have that $f$ is a surjection, since
assuming otherwise yields a contradiction.

Since $Y \setminus \{(0,0)\}$
is disconnected and each of $A$, $B$, and
$C$ is infinite, the fact that $f$ is a continuous
surjection implies there exist infinitely many
$x \in X$ such that $f(x)=(0,0)$.
Therefore, there exist
$p_0=(a,0),p_1=(b,0) \in X$ with $b>a+2k$ such that
$f(p_0)=f(p_1)=(0,0)$. Then
$g\circ f(p_0)=g \circ f(p_1) = g(0,0)$. Therefore,
at least one of $p_0$ or $p_1$ is carried by
$g \circ f$ more than $k$ steps away from itself,
contrary to the assumption that $g \circ f$ and
$1_X$ are homotopic within $k$ steps. The
assertion that $(X,(0,0))$ and $(Y,(0,0))$ do not 
have the same pointed homotopy type
follows from this contradiction. $\Box$

\begin{Example}
There exist digital images $(X,\kappa)$ and $(Y,\lambda)$
that are homotopically similar but not pointed 
homotopically similar, and that have the same long 
homotopy type but not the same pointed long homotopy 
type, and that have the same real homotopy type but
not the same pointed real homotopy type.
\end{Example}

{\em Proof}:
By ~\cite{Haarmann,BoSt}, there exist finite 
digital images $X$ and $Y$ that are homotopically 
equivalent but not pointed homotopically equivalent. By
Theorem~\ref{equiv-for-bdd}, $X$ and $Y$ are homotopically
similar but not pointed homotopically similar. By
Proposition~\ref{hequiv-implies-long}, $X$ and $Y$ have the same
long homotopy type but not the same
pointed long homotopy type. By 
Theorem~\ref{long-and-real}, $X$ and $Y$
have the same real homotopy type but not the same
pointed real homotopy type.
$\Box$

\section{Fundamental groups}
In this section, we show that digital images that are pointed 
homotopically similar, or that have the same pointed real
homotopy type, or that have the same pointed long homotopy type,
have isomorphic fundamental groups.

\begin{Theorem}
\label{Pi1}
Let $(X,x_1) \simeq^s (Y,y_1)$. Let 
$\{X_j,Y_j,f_j,g_j\}_{j=1}^{\infty}$ be as in
Definition~\ref{htpy-sim-def}, and let
$x_1 \in X_1$, $y_1 \in Y_1$.
Then there is an isomorphism 
$F: \Pi_1^{\kappa} (X,x_1) \rightarrow 
\Pi_1^{\lambda} (Y, y_1)$.
\end{Theorem}

\begin{proof}
By hypothesis, for all indices $j$ we have
that $g_j \circ f_j$ is pointed homotopic in $X_j$, 
hence in $X$, to $1_{X_j}$ 
and $f_j \circ g_j$ is pointed homotopic in $Y_j$,
hence in $Y$, to $1_{Y_j}$.

We define a function 
$F: \Pi_1^{\kappa}(X,x_1) \rightarrow
    \Pi_1^{\lambda}(Y,y_1)$ as follows.
Let $f$ be an $x_1$-based EC loop in $X$. There is a
smallest positive integer $j$ such that
the image of $f$ is contained in $X_j$. Then
$f_j \circ f$ is a $y_1$-based loop in $Y_j$.
Define $F([f]) = [f_j \circ f]$.


Suppose $f' \in [f]_X$. For some smallest indices
$j,j'$, the images of $f,f'$ lie in
$X_j,X_{j'}$, respectively. 
Further, there is some $a > j, j'$ such that $f$ and $f'$ are EC-homotopic in $X_a$. We have:
\[ F([f])= [f_j \circ f] = 
[(f_a|X_j) \circ f] =
   [f_a \circ f] = [f_a \circ f'] =
   [(f_a|X_{j'}) \circ f']=[f_{j'} \circ f'] = 
   F([f']).
\]
Therefore, $F$ is well defined.

Suppose $f$ is an $x_1$-based EC loop in $X$ such that
$F([f])=[\overline{y_1}]$, the identity element of
$\Pi_1^{\lambda}(Y,y_1)$. If $j$ is the minimal index
such that the image of $f$ is contained in $Y_j$, then
\[ F([f]) =[\overline{y_1}]= [f_j \circ f], \]
so
\[ [f]=[(g_j \circ f_j) \circ f]=
   [g_j \circ (f_j \circ f)]=
[g_j \circ \overline{y_1}]=[\overline{x_1}], \]
the identity element of $\Pi_1^{\kappa}(X,x_1)$. Therefore,
$F$ is one-to-one.

Given a $y_1$-based EC loop $g$ in $Y$, the image of $g$ is
contained in some $Y_j$ for some smallest $j$.  If 
$j' \leq j$ is such that $j'$ is the minimal index such
that the image of $g_j \circ g$ is contained in $X_{j'}$,
then
\[ [g] = [f_j \circ g_j \circ g] = 
   [f_{j'} \circ (g_j \circ g)] = F([(g_j \circ g)]). \]
Thus, $F$ is onto.

Let $L_i$ be $x_1$-based EC loops in $X$, $i \in \{0,1\}$. 
Suppose the minimal indices for the $X_j$ containing the
images of the $L_i$ are $j_0,j_1$ respectively, where,
without loss of generality, $j_0 \leq j_1$. Then
$j_1$ is the minimal index of the $X_j$ containing
$L_0 * L_1$. Then
\begin{align*}
 F([L_0 * L_1]) &= [f_{j_1} \circ (L_0 * L_1)]=
   [f_{j_1}(L_0) * f_{j_1}(L_1)]=
   [(f_{j_1}|X_{j_0})(L_0) * f_{j_1}(L_1)] \\
&= [f_{j_0} \circ (L_0)] \cdot [f_{j_1} \circ (L_1)] = 
   F(L_0) \cdot F(L_1).
\end{align*}
Therefore, $F$ is a homomorphism.  This completes
the proof.
\end{proof}

\begin{Theorem}
\label{real-Pi1}
Let $(X,x_0) \simeq_{\kappa,\lambda}^\R(Y,y_0)$. 
Then
$\Pi_1^{\kappa}(X,x_0)$ and 
$\Pi_1^{\lambda}(Y,y_0)$ are isomorphic.
\end{Theorem}

\begin{proof} The hypothesis implies that there exist
pointed continuous functions
$f: (X, x_0) \rightarrow (Y, y_0)$ and
$g: (Y, y_0) \rightarrow (X, x_0)$, and pointed
long homotopies $H: X \times [0,1] \rightarrow X$
from $g \circ f$ to $1_X$ and 
$G: Y \times [0,1] \rightarrow Y$ from
$f \circ g$ to $1_Y$.

Given an $x_0$-based EC loop $L$ in $X$, define
$f_*: \Pi_1^{\kappa}(X,x_0) \rightarrow \Pi_1^{\lambda}(Y,y_0)$
by $f_*([L])=[f \circ L]$. This function $f_*$ is the homomorphism induced by $f:X\to Y$ 
(see~\cite{BoSt} for a proof that $f_*$ is a well-defined homomorphism). It remains to show that $f_*$ is one-to-one and onto.

Suppose $f_*([L])=[f \circ L] = [\overline{y_0}]$, the identity element of
$\Pi_1^{\lambda}(Y,y_0)$. Since $g\circ f \simeq^\R 1_X$, we have $g\circ f \circ L \simeq^\R L$. Since the image of $L$ is finite, an argument similar to that used in the proof of Theorem~\ref{long-and-real} shows that $g \circ f \circ L \simeq L$ by a pointed homotopy, and thus by a pointed EC homotopy. Thus,
\[ [L] = [g \circ f \circ L] = [g \circ (f\circ L)] = [g\circ \overline{y_0}] = [\overline{x_0}], \]
the identity element of $\Pi_1^\kappa(X,x_0)$. Hence, $f_*$ is one-to-one.

Suppose $M$ is a $y_0$-based EC loop in $Y$. Then 
there is a real homotopy in $Y$ between $M$ and $f \circ g \circ M$ that 
holds the endpoints fixed.
As above, since the image of $M$ is finite, the argument from Theorem \ref{long-and-real} gives
a homotopy between $M$ and $f \circ g \circ M$ that 
holds the endpoints fixed. Therefore, 
\[ [M] = [f \circ g \circ M] = f_*([g \circ M]). \]
Thus, $f_*$ is onto.
\end{proof}

\begin{Theorem}
\label{long-Pi1}
Let $(X,x_1) \simeq_{\kappa,\lambda}^L (Y,y_1)$.
Then $\Pi_1^{\kappa}(X,x_0)$ and 
$\Pi_1^{\lambda}(Y,y_0)$ are isomorphic.
\end{Theorem}
\begin{proof}
By Theorem~\ref{long-implies-real}, $(X,\kappa,x_0)$ and $(Y,\lambda,y_0)$ have the same real homotopy type. Then Theorem \ref{real-Pi1} gives the result.
\end{proof}

\section{Wedges and Cartesian products}
In this section we show that the wedge and Cartesian product operations
preserve pointed homotopic similarity, 
pointed long homotopy type, and pointed real
homotopy type. 

Recall that
if $X_1, X_2$ are digital images in
$\Z^m$ with the same adjacency relation
$\kappa$ such that $X_1 \cap X_2 = \{x_0\}$
for some point $x_0$, and such that $x_0$ is the
only point of $X_1$ adjacent to
any point of $X_2$ and is also the
only point of $X_2$ adjacent to
any point of $X_1$,
then $X = X_1 \cup X_2$, with the
$\kappa$ adjacency, is the {\em wedge} of
$X_1$ and $X_2$, denoted $X = X_1 \wedge X_2$, and
$x_0$ is the {\em wedge point} of $X$. Also,
if $X = X_1 \wedge X_2$, $Y = Y_1 \wedge Y_2$,
$x_0$ is the wedge point of $X$, $y_0$ is the 
wedge point of $Y$,
and $f_i: (X_i,x_0) \rightarrow (Y_i,y_0)$ are
$(\kappa, \lambda)$-pointed continuous 
for $i \in \{1,2\}$, then the function 
$f_1 \wedge f_2: X \rightarrow Y$ defined by
\[ (f_1 \wedge f_2)(x) =
   \left \{ \begin{array}{ll}
      f_1(x) & \mbox{if } x \in X_1; \\
      f_2(x) & \mbox{if } x \in X_2,
      \end{array} \right .
\]
is easily seen to be $(\kappa,\lambda)$-continuous.

\begin{Theorem}
\label{wedges}
Suppose $X_1, X_2$ are digital images in
$\Z^m$, and $Y_1$ and $Y_2$ are digital images in
$\Z^n$. If
$(X_1,x_0) \simeq_{\kappa,\lambda}^s(Y_1,y_0)$
and
$(X_2,x_0) \simeq_{\kappa,\lambda}^s(Y_2,y_0)$,
and $X=X_1 \wedge X_2$ has wedge point $x_0$ and
$Y=Y_1 \wedge Y_2$ has wedge point $y_0$, then
$(X, x_0) \simeq_{\kappa,\lambda}^s (Y, y_0)$.
\end{Theorem}

\begin{proof}
By hypothesis, there are subsets $X_{j,n}$ 
of $X_j$ and $Y_{j,n}$ 
of $Y_j$, such that
$X_{j,n} \subset X_{j+1,n}$ and 
$Y_{j,n} \subset Y_{j+1,n}$, 
for $j \in \{1,2\}$, $n \in {\N}$,
$X_j = \bigcup_{n=1}^{\infty}X_{j,n}$,
$Y_j = \bigcup_{n=1}^{\infty}Y_{j,n}$,
and pointed continuous functions
$f_n: (X_{1,n},x_0) \rightarrow (Y_{1,n}, y_0)$,
$g_n: (Y_{1,n},y_0) \rightarrow (X_{1,n}, x_0)$,
$f_n': (X_{2,n},x_0) \rightarrow (Y_{2,n}, y_0)$,
$g_n': (Y_{2,n},y_0) \rightarrow (X_{2,n}, x_0)$,
such that 
$g_n \circ f_n$ is pointed homotopic in $X_{1,n}$ to
$1_{X_{1,n}}$,
$f_n \circ g_n$ is pointed homotopic in $Y_{1,n}$ to
$1_{Y_{1,n}}$,
$g_n' \circ f_n'$ is pointed homotopic in $X_{2,n}$ to
$1_{X_{2,n}}$, and
$f_n' \circ g_n'$ is pointed homotopic in $Y_{2,n}$ to
$1_{Y_{2,n}}$. Also, $x_0$ is the wedge point for
each $X_{1,n} \wedge X_{2,n}$, and $y_0$ is the 
wedge point for each $Y_{1,n} \wedge Y_{2,n}$.

Then it is easily seen that
$(g_n \wedge g_n') \circ (f_n \wedge f_n')$ is pointed homotopic
in $X_1 \wedge X_2$ to $1_{X_1 \wedge X_2}$ and
$(f_n \wedge f_n') \circ (g_n \wedge g_n')$ is pointed homotopic
in $Y_1 \wedge Y_2$ to $1_{Y_1 \wedge Y_2}$. The assertion
follows. 
\end{proof}

\begin{Theorem}
\label{long-wedges}
Suppose $X_1, X_2$ are digital images in
$\Z^m$, and $Y_1$ and $Y_2$ are digital images in
$\Z^n$. If
$(X_1,x_0) \simeq_{\kappa,\lambda}^L (Y_1,y_0)$,
and
$(X_2,x_0) \simeq_{\kappa,\lambda}^L (Y_2,y_0)$,
and $X=X_1 \wedge X_2$ has wedge point $x_0$ and
$Y=Y_1 \wedge Y_2$ has wedge point $y_0$, then
$(X, x_0) \simeq_{\kappa,\lambda}^L (Y, y_0)$.
\end{Theorem}

{\em Proof}:
By hypothesis, for $i \in \{1,2\}$ 
there exist pointed continuous
functions $f_i: (X_i, x_0) \rightarrow (Y_i,y_0)$ and
$g_i: (Y_i,y_0) \rightarrow (X_i,x_0)$,
long pointed homotopies
$H_i: (X_i,x_0) \times \Z \rightarrow (X_i,x_0)$ from
$1_{X_i}$ to $g_i \circ f_i$ in $X_i$, and
long pointed homotopies
$K_i: (Y_i,y_0) \times \Z \rightarrow (Y_i,y_0)$ from
$1_{Y_i}$ to $f_i \circ g_i$ in $Y_i$.

Then the function
$H: (X_1 \wedge X_2, x_0) \times \Z \rightarrow 
    (X_1 \wedge X_2, x_0)$ defined by
\[ H(x,t) = \left \{ \begin{array}{ll}
             H_1(x,t) & \mbox{if } x \in X_1; \\
             H_2(x,t) & \mbox{if } x \in X_2,
             \end{array} \right .
\]
is a pointed long homotopy in $X$ from $1_{X_1 \wedge X_2}$ to
$(g_1 \wedge g_2) \circ (f_1 \wedge f_2)$.
Similarly, the function
$K: (Y_1 \wedge Y_2, x_0) \times \Z \rightarrow 
    (Y_1 \wedge Y_2, x_0)$ defined by
\[ K(y,t) = \left \{ \begin{array}{ll}
             K_1(y,t) & \mbox{if } y \in Y_1; \\
             K_2(y,t) & \mbox{if } y \in Y_2,
             \end{array} \right .
\]
is a pointed long homotopy in $Y$ from $1_{Y_1 \wedge Y_2}$ to
$(f_1 \wedge f_2) \circ (g_1 \wedge g_2)$. The assertion
follows. $\Box$

Arguments similar to those above demonstrate the following.

\begin{Theorem}
\label{real-wedges}
Suppose $X_1, X_2$ are digital images in
$\Z^m$, and $Y_1$ and $Y_2$ are digital images in
$\Z^n$. If
$(X_1,\kappa,x_0) \simeq_{\kappa,\lambda}^\R 
(Y_1,\lambda,y_0)$ and
$(X_2,x_0) \simeq_{\kappa,\lambda}^\R (Y_2,y_0)$,
and $X=X_1 \wedge X_2$ has wedge point $x_0$ and
$Y=Y_1 \wedge Y_2$ has wedge point $y_0$, then
$(X, x_0) \simeq_{\kappa,\lambda}^\R (Y, y_0)$. $\Box$
\end{Theorem}

Now we consider Cartesian products. For our pointed assertions in the following, we
assume
\[ x_i=(x_{i,1},x_{i,2},\ldots,x_{i,n_i})
\in X_i \subset \Z^{n_i}, 
~~~~~y_i=(y_{i,1},y_{i,2},\ldots,y_{i,n_i})
\in Y_i \subset \Z^{n_i}, \]
\[ x_0=(x_1,x_2,\ldots,x_{n_i})=
 (x_{1,1},x_{1,2},\ldots,x_{1,n_1},x_{2,1},x_{2,2},\ldots,x_{2,n_2}, \ldots, x_{k,1},x_{k,2},\ldots,x_{k,n_k}) \in 
 \Pi_{i=1}^k X_i \subset \Z^D,
\]
\[ y_0=(y_1,y_2,\ldots,y_{n_i})=
 (y_{1,1},y_{1,2},\ldots,y_{1,n_1},y_{2,1},y_{2,2},\ldots,y_{2,n_2}, \ldots, y_{k,1},y_{k,2},\ldots,y_{k,n_k}) \in 
 \Pi_{i=1}^k Y_i \subset \Z^D,
\]
where $D=\sum_{i=1}^k n_i$.

\begin{Theorem}
\label{product-thm}
Let $X_i$ and $Y_i$ be digital images in
$(\Z^{n_i}, c_{n_i})$, $i \in \{1,2,\ldots,k\}$.
Let $x_i \in X_i$, $y_i \in Y_i$. Let
$D=\sum_{i=1}^k n_i$.
\begin{itemize}
\item If $X_i \simeq_{c_{n_i},c_{n_i}} Y_i$
      for $i \in \{1,2,\ldots,k\}$, then
      $\Pi_{i=1}^k X_i \simeq_{c_D,c_D} 
      \Pi_{i=1}^k Y_i$.
      If $(X_i,x_i) \simeq_{c_{n_i},c_{n_i}}
      (Y_i,y_i)$ for 
      $i \in \{1,2,\ldots,k\}$, then
      $(\Pi_{i=1}^k X_i,x_0) \simeq_{c_D,c_D} 
      (\Pi_{i=1}^k Y_i,y_0)$.
\item If $X_i \simeq^s_{c_{n_i},c_{n_i}} Y_i$
      for $i \in \{1,2,\ldots,k\}$, then
      $\Pi_{i=1}^k X_i \simeq_{c_D,c_D}^s 
      \Pi_{i=1}^k Y_i$.
      If $(X_i,x_i) \simeq_{c_{n_i},c_{n_i}}^s 
      (Y_i,y_i)$ 
      for $i \in \{1,2,\ldots,k\}$, then
      $(\Pi_{i=1}^k X_i,x_0) \simeq_{c_D,c_D}^s 
      (\Pi_{i=1}^k Y_i,y_0)$.
\item If $X_i \simeq_{c_{n_i},c_{n_i}}^L Y_i$
      for $i \in \{1,2,\ldots,k\}$, then
      $\Pi_{i=1}^k X_i \simeq_{c_D,c_D}^L 
      \Pi_{i=1}^k Y_i$.
      If $(X_i,x_i) \simeq_{c_{n_i},c_{n_i}}^L 
      (Y_i,y_i)$ 
      for $i \in \{1,2,\ldots,k\}$, then
      $(\Pi_{i=1}^k X_i,x_0) \simeq_{c_D,c_D}^L 
      (\Pi_{i=1}^k Y_i,y_0)$.
\item If $X_i \simeq_{c_{n_i},c_{n_i}}^\R Y_i$
      for $i \in \{1,2,\ldots,k\}$, then
      $\Pi_{i=1}^k X_i \simeq_{c_D,c_D}^\R 
      \Pi_{i=1}^k Y_i$.
      If $(X_i,x_i) \simeq_{c_{n_i},c_{n_i}}^\R 
      (Y_i,y_i)$ 
      for $i \in \{1,2,\ldots,k\}$, then
      $(\Pi_{i=1}^k X_i,x_0) \simeq_{c_D,c_D}^\R 
      (\Pi_{i=1}^k Y_i,y_0)$.
\end{itemize}
\end{Theorem}

\begin{proof}
We give proofs for the unpointed assertions. In
all cases, the proof of the pointed assertion is
virtually identical to that for its unpointed
analog. We let 
$X = \Pi_{i=1}^k X_i \subset \Z^D$,
$Y = \Pi_{i=1}^k Y_i \subset \Z^D$.

First we prove the statement about ordinary homotopy equivalence.
Suppose $X_i \simeq_{c_{n_i},c_{n_i}} Y_i$
      for $i \in \{1,2,\ldots,k\}$. Then there
      exist $(c_{n_i},c_{n_i})$-continuous
      functions $f_i: X_i \to Y_i$ and
      $g_i: Y_i \to X_i$, and homotopies
      $H_i: X_i \times [0,u_i]_{\Z} \to X_i$
      from $1_{X_i}$ to $f_i \circ g_i$ and
      $K_i: Y_i \times [0,v_i]_{\Z} \to Y_i$ 
      from $1_{Y_i}$ to $g_i \circ f_i$. Without
      loss of generality, we can replace
      each $u_i$ and each $v_i$ with
      $U = \max\{u_1,v_1, \ldots, u_k,v_k\}$, since
      if $u_i < U$ then we can extend $H_i$ by
      defining $H_i(x,t)=H_i(x,u_i)=
      g_i\circ f_i(x)$ for $u_i \leq t \leq U$, and
      similarly for $K_i$.
      
      For $a_i \in X_i$, let $f: X \to Y$ be
      defined by 
      \[ f(a_1,a_2, \ldots, a_k) =
         (f_1(a_1),f_2(a_2), \ldots, f_k(a_k)). \]
      For $b_i \in Y_i$, let $g: Y \to X$ be
      defined by 
      \[ g(b_1,b_2, \ldots, b_k) =
         (g_1(b_1),g_2(b_2), \ldots, g_k(b_k)). \]
      Let $H: X \times [0,U]_{\Z} \to X$ be
      defined by
      \[ H(a_1,a_2, \ldots, a_k,t) =
         (H_1(a_1,t),H_2(a_2,t), \ldots,
          H_k(a_k,t)).
      \]
      Let $K: Y \times [0,U]_{\Z} \to Y$ be
      defined by
      \[ K(b_1,b_2, \ldots, b_k,t) =
         (K_1(b_1,t),K_2(b_2,t), \ldots,
          K_k(b_k,t)).
      \]
      It is easy to see that $H$ is a 
      $(c_D,c_D)$-homotopy from $1_X$ to
      $g \circ f$, and $K$ is a 
      $(c_D,c_D)$-homotopy from $1_Y$ to
      $f \circ g$.
      
Minor modifications in the argument
given above allow us to demonstrate the claims for long homotopy type and real homotopy type. 

It remains to prove the statement about homotopic similarity.
 Suppose $X_i \simeq_{c_{n_i},c_{n_i}}^s Y_i$
      for $i \in \{1,2,\ldots,k\}$. Then there
      exist $\{X_{i,j}\}_{j=1}^{\infty} \subset X_i$ such that $X_{i,j} \subset X_{i,j+1}$ and
      $\bigcup_{j=1}^{\infty}X_{i,j}=X_i$,
      $\{Y_{i,j}\}_{j=1}^{\infty} \subset Y_i$ such that $Y_{i,j} \subset Y_{i,j+1}$ and
      $\bigcup_{j=1}^{\infty}Y_{i,j}=Y_i$,
      continuous functions 
      $f_{i,j}: X_{i,j} \to Y_{i,j}$ and
      $g_{i,j}: Y_{i,j} \to X_{i,j}$, and 
      homotopies
      $H_{i,j}: X_{i,j} \times [0,u_{i,j}]_{\Z}
      \to X_{i,j}$ from $g_{i,j} \circ f_{i,j}$ to
      $1_{X_{i,j}}$ and
      $K_{i,j}: Y_{i,j} \times [0,v_{i,j}]_{\Z}
      \to Y_{i,j}$ from $f_{i,j} \circ g_{i,j}$ to
      $1_{Y_{i,j}}$. As above, for each $j$ we
      can replace each $u_{i,j}$ and
      each $v_{i,j}$ by 
      $U_j = \max\{u_{i,j},v_{i,j}\}_{i=1}^k$.
      
      Notice that for all $j$,
      $\Pi_{i=1}^k X_{i,j} \subset \Pi_{i=1}^k X_{i,j+1}$ and 
      $\Pi_{i=1}^k Y_{i,j} \subset \Pi_{i=1}^k Y_{i,j+1}$. Also, 
      $\bigcup_{j=1}^{\infty} \Pi_{i=1}^k X_{i,j} =
      X$ and
      $\bigcup_{j=1}^{\infty} \Pi_{i=1}^k Y_{i,j} =
      Y$.
      
      Let $f_j: \Pi_{i=1}^k X_{i,j} \to \Pi_{i=1}^k Y_{i,j}$ be
      defined by
      \[f_j(a_1,\ldots,a_k)=(f_{1,j}(a_1), \ldots,
         f_{k,j}(a_k))
      \]
      for $a_i \in X_{i,j}$.
      Let $g_j: \Pi_{i=1}^k Y_{i,j} \to \Pi_{i=1}^k X_{i,j}$ be
      defined by
      \[g_j(b_1,\ldots,b_k)=(g_{1,j}(b_1), \ldots,
         g_{k,j}(b_k))
      \]
      for $b_i \in Y_{i,j}$.
      Let $H_j: \Pi_{i=1}^k X_{i,j} \times [0,U_j]_{\Z} \to \Pi_{i=1}^k X_{i,j}$ be
      defined by 
      \[ H_j(a_1,\ldots,a_k,t)=(H_{1,j}(a_1,t), \ldots, H_{k,j}(a_k,t)), \]
      for $a_i \in X_{i,j}$.
      Let $K_j: \Pi_{i=1}^k Y_{i,j} \times [0,U_j]_{\Z} \to \Pi_{i=1}^k Y_{i,j}$ be
      defined by 
      \[ K_j(b_1,\ldots,b_k,t)=(K_{1,j}(b_1,t), \ldots, K_{k,j}(b_k,t)), \]
      for $b_i \in Y_{i,j}$.
      Then it is easily seen that $H_j$ is a
      homotopy from $f_j \circ g_j$ to $1_{X_j}$,
      and $K_j$ is a homotopy from $g_j \circ f_j$
      to $1_{Y_j}$. Therefore, $X \simeq^s Y$.
\end{proof}

\section{Further remarks and open questions}
We have introduced three notions of digital images 
having homotopic resemblance -
homotopic similarity, having the same long 
homotopy type, and having the same real
homotopy type - in both unpointed and pointed versions.
Unlike the usual definition of digital homotopy 
equivalence, these let us consider two digital 
images $X$ and $Y$ as similar
with respect to homotopy properties even if one of them 
has a component with infinite diameter
and the other does not. We have shown that two digital images
that are homotopy equivalent are homotopically similar,
have the same long homotopy type, and have the same
real homotopy type, and that
the converses hold when both images are finite; 
however, we have shown the converses to be 
false if one of the images has infinite diameter. 
We have shown that two digital images that
share any of these three pointed homotopy
resemblances have isomorphic fundamental groups. 
We have also shown that wedges preserve pointed 
homotopy similarity, pointed long homotopy type, 
and pointed real homotopy type; as do finite
Cartesian products when we use relaxed adjacencies.

In addition to several questions
stated earlier that we have not answered at this writing,
we have the following.

\begin{Question}
\label{which-implies}
(Unpointed and pointed versions:) 
      Which, if any, of homotopic similarity,
      having the same long homotopy type,
      and having the same real homotopy type,
      implies either of the others?
\end{Question}
Corollary~\ref{long-implies-real-type} is our only result
concerning Question~\ref{which-implies},
      that having the same long homotopy type implies
      having the same real homotopy type.
\begin{Question}
\label{which-pair}
(Unpointed and pointed versions:) 
Which, if any, of these relations are equivalent?
\end{Question}
As above, a negative example for Question~\ref{which-pair} 
would require a pair of digital images $(X,Y)$ in which at
least one of the members is infinite.

\section{Acknowledgment}
The suggestions of an anonymous
reviewer were quite helpful and are much appreciated.

\end{document}